\documentclass[11pt]{article}
\usepackage{enumerate}
\usepackage{amsmath}
\usepackage{amssymb,latexsym}
\usepackage{amsthm}
\usepackage{color}
\usepackage[normalem]{ulem}
\usepackage{graphicx}
\usepackage{mathabx}
\usepackage{hyperref}
\usepackage{amscd}
\usepackage{todonotes}

\usepackage{comment}
\usepackage{adjustbox}
\usepackage{tikz}
\usepackage{caption}
\usepackage{pgfplots}
\newcommand{\vgeq}{\mathrel{\rotatebox{90}{$\geq$}}}

\numberwithin{table}{section}

\title{Complete $3$-term arithmetic progression free sets of small size in vector spaces and other abelian groups}
\author{Bence Csajb\'ok, \footnote{Dipartimento di Meccanica, Matematica e Management, Politecnico di Bari, Via Orabona 4, I-70125 Bari, Italy. Current address: Department of Computer Science, ELTE Eötvös Loránd University, Budapest, Hungary. E-mail: {\tt bence.csajbok@ttk.elte.hu}}\quad Zolt\'an L\'or\'ant Nagy \footnote{ELTE Linear Hypergraphs  Research Group, ELTE Eötvös Loránd University, Budapest, Hungary. The author is supported by the Hungarian Research Grant (NKFIH) No. PD  134953. and No. K. 124950 and the University Excellence Fund of Eötvös Loránd University E-mail: {\tt nagyzoli@cs.elte.hu}}}
\date{}

\newcommand{\Z}{{\mathbb Z}}

\newcommand{\cP}{{\mathcal P}}

\newcommand{\cL}{{\mathcal L}}

\newcommand{\F}{{\mathbb F}}

\newcommand{\E}{\mathbb{E}}
\newcommand{\pP}{\mathbb{P}}
\newcommand{\sat}{\mathrm{sat}}

\newtheorem{theorem}{Theorem}[section]
\newtheorem{problem}[theorem]{Problem}
\newtheorem{lemma}[theorem]{Lemma}
\newtheorem{corollary}[theorem]{Corollary}
\newtheorem{definition}[theorem]{Definition}
\newtheorem{proposition}[theorem]{Proposition}

\newtheorem{example}[theorem]{Example}

\newtheorem{observ}[theorem]{Observation}
\newtheorem{remark}[theorem]{Remark}
\newtheorem{constr}[theorem]{\bf Construction}
\newtheorem{cor}[theorem]{\bf Corollary}

\newtheorem{prop}[theorem]{\bf Proposition}

\theoremstyle{definition}

\DeclareMathOperator{\PG}{{PG}}
\DeclareMathOperator{\AG}{{AG}}

\begin{document}
	\maketitle
	
	\begin{abstract}
		
		
		A subset $S$ of an abelian group $G$ is called $3$-$\mathrm{AP}$ free if it does not contain a three term arithmetic progression. Moreover, $S$ is called complete $3$-$\mathrm{AP}$ free, if it is maximal w.r.t. set inclusion.
		One of the most central problems in additive combinatorics is to determine the maximal size of a $3$-$\mathrm{AP}$ free set, which is necessarily complete. 
		In this paper we are interested in the minimum size of complete $3$-$\mathrm{AP}$ free sets. We define and study saturation w.r.t. $3$-$\mathrm{AP}$s and present constructions of small complete $3$-$\mathrm{AP}$ free sets 
		and $3$-$\mathrm{AP}$ saturating sets 
		for several families of vector spaces and cyclic groups.
		
	\end{abstract}
	
	
	
	\section{Introduction}

	Studying the maximum possible size of a subset of a vector space  over a finite field which contain either no (non-trivial) solution to a given linear equation or not too many collinear points is a classical yet vibrant research area \cite{Alon, Bruyn,Eber, EG, KovacsNagy, FGR, Mimura, Lisa, Ilja}.
	The most notable examples are the Sidon sets and the so-called cap-set problem. The latter one is to determine the largest subset of $\F_3^n$ containing no complete line, or in other terms, no arithmetic progression of length 3 ($3$-$\mathrm{AP}$), or no collinear triplets. In general,  point sets of finite affine or projective spaces with no three in line are called caps. Recently Ellenberg and Gijswijt proved a breakthrough result regarding caps in $\F_3^n$ \cite{EG} building on the ideas of Croot, Lev and Pach \cite{CLP}, and they also proved that the size of a 
	$3$-$\mathrm{AP}$ free set of $\F_p^n$ is always bounded from above by $(p-\delta_p)^n$ for some small constant $\delta_p>0$ depending only on $p$, see also \cite{EPPP}.
	
	For the general case of abelian groups,  the maximum size of $3$-$\mathrm{AP}$ free sets  was discussed in the classical paper of Frankl, Graham and Rödl \cite{FGR}.   A nice and general overview on additive combinatorial and extremal results concerning arithmetic progressions is by Shkredov \cite{Ilja}.
	
	A finite point set (with respect to a property)  is called  \textit{complete} if it does not contained as a subset in a larger point set, satisfying the same property.
	In extremal problems described above, usually constructions of maximum size are in the center of attention. These are complete by definition.  However, in several cases, the whole spectra of sizes matters for complete structures, and the structure of smallest size in particular. For example, in the case of complete caps over $\PG(n,q)$, the point set is corresponding to the parity check matrix of a $q$-ary linear code with codimension $n+1$,
	Hamming distance $4$, and covering radius $2$, see \cite{Hirs}. 
	
	In this paper we investigate the less studied lower end of the spectrum of possible sizes of complete $3$-$\mathrm{AP}$ free sets, the minimum size. We discuss the minimum size in arbitrary abelian groups of odd order and highlight the case of finite vector spaces.


	

	A $3$-term arithmetic progression of the abelian group $G$, $3$-$\mathrm{AP}$ for short,  is a set of three distinct elements of $G$ of the form $g,g+d,g+2d$, where $g,d\in G$. We will call $d$ the difference. If $d$ is the difference of a $3$-$\mathrm{AP}$ in $G$, then the order of $d$ is at least $3$. Let $\F_q$ denote the Galois field of $q$ elements, and $o_q(a)$ denotes the multiplicative order of $a \in\F_q$. In order to have a $3$-$\mathrm{AP}$ in $\F_q^n$ we need $q$ to be odd, so we will only consider this case. $A+B$ denotes the sumset $\{a+b :  a\in A, b\in B\}$ of sets $A$ and $B$,  while $A\dot+B$ denotes the restricted sumset where the summands must be distinct.

	\begin{definition}
		$A\subseteq G$ is called \textbf{$3$-$\mathrm{AP}$-free} if it does not contain a $3$-$\mathrm{AP}$ of $G$.
		Moreover, $A$ is \textbf{complete $3$-$\mathrm{AP}$-free} if it is $3$-$\mathrm{AP}$ free and  not contained in a larger $3$-$\mathrm{AP}$ free set. 
	\end{definition}
	\vspace{-0.2cm}
	The completeness of a $3$-$\mathrm{AP}$ free set can be interpreted via saturation as well: $S$ is complete $3$-$\mathrm{AP}$ free if $S$ is $3$-$\mathrm{AP}$ free and a saturating set w.r.t. $3$-$\mathrm{AP}$s.

	\begin{definition}
		For a subset $S\subseteq G$ we say that $S$ is\textbf{ $3$-$\mathrm{AP}$ saturating } or in other words, $S$ \textbf{$3$-$\mathrm{AP}$ saturates } $G$  if for each $x \in G \setminus S$  there is a $3$-$\mathrm{AP}$ of $G$ consisting of $x$ and two  elements of $S$. 
		In a broader context, we say that $S$ \textbf{$3$-$\mathrm{AP}$ saturates } 
		a set $H\subset G$ if similar condition holds for the elements of $H\setminus S$.
	\end{definition}
	\vspace{-0.2cm}
	In this paper we will mostly consider the problem of $3$-$\mathrm{AP}$ saturation in groups of odd order. 
	\vspace{-0.2cm}
	\begin{definition}
		For $g\in G$ if $\ell\in \mathbb{Z}$ is positive, then $\ell g= \underbrace{g+g+\ldots+g}_{\ell} \in G$ and $(-\ell)g=~-~(\ell g)\in ~G$. 
		If the order of $G$ is odd then we define $\frac12 g$ as the unique element $x\in G$ such that $x+x=g$, that is,  $x=(k+1)g$, where the order of $g$ is $2k+1$.
	\end{definition}
	
	
	\begin{observ}\label{obs1}
		A set $A\subseteq G$ is $3$-$\mathrm{AP}$ saturating, iff  for every $x \in G \setminus A$  there is a $3$-$\mathrm{AP}$  consisting of $x$ and $\{a_1, a_2\}\subseteq A$ such that $(i)$ either  $x=2a_1-a_2$, or $(ii)$ $2x=a_1+a_2$.
		
		If $G$ has odd order then $(ii)$ is equivalent to   
		$x=\frac12 a_1 + \frac12 a_2$.
	\end{observ}
	\vspace{-0.3cm}
	Saturation and completeness with respect to  $3$-$\mathrm{AP}$s  were considered in integer sequences as well, see \cite{SandorCs, Fang}.\\
	The postage stamp problem seeks the greatest integer  $r=r_k$ such that  there exists a set $A_k$ of $k$ positive integers together with $0$ such that \vspace{-0.3cm} $$i\in A_k+A_k \ \ \mbox{ for all \  } i=0, 1, \ldots, r.$$ 
	Mrose \cite{Mrose} and independently,  Fried \cite{Fried}  showed that the  $r_k$ is at least $\frac{2}{7}k^2+O(k)$. 
	A set of non-negative integers $A$ is called a basis of order two if $A+A=\mathbb{N}$ holds for the sumset. 
	Note that a variation of Mrose's construction can be extended to an additive $2$-basis of $\mathbb{N}$, see \cite{Habsi, Hof}. Such constructions lead to 
	small sets $S$ of $\F_p$ for which every element of $\F_p$ is an arithmetic mean of a pair of elements from $S$, hence $S$ saturates the $3$-$\mathrm{AP}$s. We will discuss this in Section 4.
	
	These constructions provide further motivations to introduce the saturation with respect to a set of (coefficient) vectors. 
	
	

	
	For a subset $S$ of elements of a group  $G$ (written additively) we will write $S^*$ to denote $S$ minus the neutral element. 
	
	\begin{definition}[$W$-avoiding and $W$-saturating sets]
		\label{defavoid}
		Let $W$ denote a set of vectors from $\F_q^*\times \F_q^*$. 
		We define $W$-avoiding and $W$-saturating sets in $\F_q^n$ as follows. 
		
		$A\subseteq \F_q^n$ is  \textbf{$W$-avoiding}, if there is no $w=(\lambda_1, \lambda_2)\in W$ such that $a=\lambda_1a'+\lambda_2a''$ has a non-trivial solution 
		with $a,a',a''$ pairwise distinct vectors of $A$. 

		$A\subseteq \F_q^n$ is  \textbf{$W$-saturating} in $\F_q^n$ if for each    $x\in \F_q^n\setminus A$
		there exists a $w=(\lambda_1, \lambda_2)\in W$ such that $x=\lambda_1a'+\lambda_2a''$ for a pair $(a', a'')\in A^2$, $a'\neq a''$. 
		
		$A \subseteq \F_q^n$ is \textbf{complete $W$-avoiding} if it is $W$-avoiding and $W$-saturating. 
		
	\end{definition}

	If $W$ consists of a single vector $W=\{w\}$, we omit the brackets for brevity.

	\begin{definition}[Avoiding and saturating sets in groups]
		
		Let $W$ denote a subset of $\mathbb{Z}\times \mathbb{Z}$.   
		We define $W$-avoiding, $W$-saturating and complete $W$-avoiding sets in  abelian groups $G$ similarly to Definition \ref{defavoid}. We will be mostly interested in the cases when $W$ is an element of $\{(2,-1), (1,1),(1,-1)\}$. 
		
		If $G$ has odd order then we also define $(\frac12,\frac12)$-saturating sets.
		
	\end{definition}
	
	\begin{remark}
		\label{r1}
		For a subset $A\subseteq \F_q^n$ and for $(\lambda_1,\lambda_2)\in \F_q^*\times \F_q^*$ the following properties are equivalent: $(i)$ $A$ is $(\lambda_1,\lambda_2)$-avoiding, $(ii)$ $A$ is $(1/\lambda_1,-\lambda_2/\lambda_1)$-avoiding, $(iii)$ $A$ is $(1/\lambda_2,-\lambda_1/\lambda_2)$-avoiding.
		
		In a group $G$ the following properties are equivalent: $(i)$ $A$ is $3$-$\mathrm{AP}$ free $(ii)$ $A$ is $(2,-1)$-avoiding, $(iii)$ there is no three pairwise distinct elements $x,y,z\in A$ such that $2x=y+z$. (If the order of $G$ is odd, then it is equivalent to say that $A$ is $(\frac12, \frac12)$-avoiding.)
		
		Similarly, for $A \subseteq G \setminus \{0\}$ the following properties are equivalent in any abelian group $G$: $(i)$ $A$ is restricted sum-free, i.e., $(A\dot+A)\cap A= \emptyset$, $(ii)$ $A$ is $(1,1)$-avoiding, $(iii)$ $A$ is $(1,-1)$-avoiding.
	\end{remark}
	
	
	\vspace{-0.3cm}
	Note that Definition \ref{defavoid} can be extended naturally to any set of vectors of $\bigcup_{t=2}^\infty (\F_q^*)^t$.
	
	Our main (but not only) focus will be the case $W=\{(2,-1)\}$
	for its correspondence to  $3$-$\mathrm{AP}$s, see Observation \ref{obs1}, and in general, the case when $W$ consists of a single vector.

	We will also apply the fact that the field $\F_{q^n}$ is itself a vector space over $\F_{q^h}$ for  $h \mid n$. This will enable us to alter the dimension of the underlying structure at times, which will provide improvements on the estimates. If $q=p^r$, $p$ prime, and $W \subseteq \F_p^* \times \F_p^*$, then studying $W$-saturating and $W$-avoiding sets in the vector space $\F_q^n$ is equivalent to study the same questions in the elementary abelian group $\F_p^{rn}$.
	
	Now we introduce the functions we wish to study.
	
	\begin{definition}
		For a group $G$ we define the following: \vspace{-0.2cm}
		\begin{enumerate}[\rm(1)]\vspace{-0.2cm}
			\item  Let $a(3-\mathrm{AP}, G)$ denote the minimum size of a complete $3$-$\mathrm{AP}$ free set of $G$.\vspace{-0.2cm}
			\item  Let $\sat(3-\mathrm{AP}, G)$ denote the minimum size of a $3$-$\mathrm{AP}$ saturating set of $G$.\vspace{-0.2cm}
			\item  Let $a(W, G)$ denote the minimum size of a complete $W$-avoiding set of $G$. If there is no $W$-avoiding set of $G$, then put $a(W,G)=\infty$.\vspace{-0.2cm}
			\item  Let $\sat(W, G)$ denote the minimum size of a $W$-saturating set of $G$.\vspace{-0.2cm}
		\end{enumerate}
	\end{definition}
	
	Observe that in $\mathbb{Z}_5$ there are no complete $(2,-1)$-avoiding sets.
	
	\begin{example} As an example, we show complete $3$-AP sets for $G=\F_3^2$ and $G=\F_5^2$ in Figure \ref{fig:kicsipelda}. These are of minimum size, c.f. Lemma \ref{satbound}. We note in advance that in $\F_q^2$ we can always find complete $3$-AP sets of size $q$ when $-2$ is not a square element in $\F_q$, c.f. Theorem \ref{main1}.
	\end{example}

	\begin{figure}[htbp] 
		\centering 
		
		\tikzstyle{small dot}=[circle, fill, inner sep=1.5pt]
		\tikzstyle{large dot}=[circle, fill, inner sep=3pt] 
		
		\begin{minipage}{0.45\textwidth} 
			\centering
			\begin{tikzpicture}
				\def\maxCoord{2}
				
				\draw[step=1, gray, very thin] (0,0) grid (\maxCoord,\maxCoord);
				
				\foreach \x in {0,...,\maxCoord} {
					\foreach \y in {0,...,\maxCoord} {
						\node[small dot] at (\x,\y) {};
					}
				}
				
				\node[large dot] at (0,0) {};
				\node[large dot] at (0,1) {};
				\node[large dot] at (1,1) {};
				\node[large dot] at (1,0) {};
				
			\end{tikzpicture}
		\end{minipage}
		\hfill 
		\begin{minipage}{0.45\textwidth} 
			\centering
			\begin{tikzpicture}
				\def\maxCoord{4}
				
				\draw[step=1, gray, very thin] (0,0) grid (\maxCoord,\maxCoord);
				
				\foreach \x in {0,...,\maxCoord} {
					\foreach \y in {0,...,\maxCoord} {
						\node[small dot] at (\x,\y) {};
					}
				}
				
				\node[large dot] at (0,0) {};
				\node[large dot] at (1,1) {};
				\node[large dot] at (2,4) {};
				\node[large dot] at (3,4) {};
				\node[large dot] at (4,1) {};
				
			\end{tikzpicture}
		\end{minipage}
		
		\caption{Complete $3$-AP free sets of minimum size in $\F_3^2$ and in $\F_5^2$.}
		\label{fig:kicsipelda}
	\end{figure}
	
	\begin{remark}\label{rem}
		Any  complete $3$-$\mathrm{AP}$ free set is $3$-$\mathrm{AP}$ saturating, any complete $W$-avoiding set is $W$-saturating by definition, thus 
		\[a(W, G)\geq \sat(W, G).\]  
		
		Since 
		$(2,-1)$-saturating sets are clearly  satisfying the $3$-$\mathrm{AP}$ saturating property in view of Observation \ref{obs1}, we have
		\begin{equation}\label{eq:triv}
			\begin{array}{ccc}
				\sat( (2,-1), G) & \geq & \sat(3-\mathrm{AP}, G)\\
				& & \\
				\vgeq &  & \vgeq\\
				& & \\
				a((2,-1),G) & \geq & a(3-\mathrm{AP}, G)\\
			\end{array}
		\end{equation}
	\end{remark}
	
	\vspace{-0.3cm}
	Our main results are as follows.

	\begin{theorem}\label{main3AP}
		Let $p$ be an odd prime and $k$ a positive integer. Then we have
		\begin{enumerate}[\rm(1)] \vspace{-0.2cm}
			\item 
			\[\sqrt{2/3}\cdot p^{2k-1}<a(3-\mathrm{AP}, \F_p^{4k-2})\leq p^{2k-1},\]  provided that  $-2$ is not a square element  in $\F_p$.  Also, \vspace{-0.2cm}
			\item 
			\[\sqrt{2/3}\cdot p^{n/2}<a(3-\mathrm{AP}, \F_p^{n})\leq a((2,-1),\F_p)^n.\]
		\end{enumerate}
	\end{theorem}
	\vspace{-0.3cm}
	Observe that $(2)$ of Theorem \ref{main3AP} motivates the study of 
	$a((2,-1),\F_p)$, or in general $a((2,-1),\Z_m)$, where $\Z_m$ is the cyclic group of order $m$.


	\begin{theorem}\label{maingroups} Let $m$ denote a positive odd integer. Then
		\begin{enumerate}[\rm(1)] \vspace{-0.2cm}
			\item     $\sat((2,-1), \mathbb{Z}_m)<\sqrt{c_m\cdot m}$ where $c_m\in [1,3]$ is a constant depending only on $m$.
			\vspace{-0.2cm}
			\item 
			$a((2,-1), \mathbb{Z}_m)<\sqrt{c_m \cdot m}$ for $c_m\in [1, 1.5]$, a constant depending only on $m$, provided that $\frac23(4^n-1) < m<4^n$ holds for some positive integer $n$.\vspace{-0.2cm}
			\item $\sqrt{2m}-0.5<\sat((1/2, 1/2), \mathbb{Z}_m)\leq (\sqrt{3.5}+o(1))\sqrt{m}\approx 1.87\sqrt{m}.$
			\vspace{-0.2cm}
			\item  $a((2,-1), \mathbb{Z}_m)=\lceil\sqrt{m}\rceil$, provided that $m$ is if form $m=2^{2t}+2^t+1$ for some positive integer $t$.
		\end{enumerate}
	\end{theorem}
	\vspace{-0.3cm}
	The close connection between the $\sat$ function and the size of the complete $3$-AP-free sets, see Remark \ref{rem}, motivates the theorem below.
	
	\begin{theorem}\label{mainvectorspace}
		Let $3<p$ be a prime and $k$ a positive integer. Then we have
		\begin{enumerate}[\rm(1)] \vspace{-0.2cm}
			\item \[\sqrt{2/3}\cdot p^{k}<\sat(3-\mathrm{AP}, \F_p^{2k})\leq \left(\frac43+\frac{r}{3\cdot o_{p}(-2)}\right)(p^k-1),\]
			where $r$ is the residue modulo $3$ of the order $o_{p}(-2)$ of $-2$ in $\mathbb{F}_{p}^{\times}$.\vspace{-0.2cm}
			\item
			\[\sqrt{2/3}\cdot p^{k+\frac12}<\sat(3-\mathrm{AP}, \F_p^{2k+1})\leq \frac23(p^{k+1}+p^k-2)+\frac{r(p^{k+1}+p^k-2)}{3\cdot o_{p}(-2)},\]  
			where $r$ is the residue modulo $3$ of the order $o_{p}(-2)$ of $-2$ in $\mathbb{F}_{p}^{\times}$.
			\vspace{-0.2cm}
			\item 
			\[ p^{k}<\sat((2, -1), \F_p^{2k})\leq \left(\frac32+\frac{r}{2\cdot o_{p}(-2)}\right)(p^k-1),\]  
			where $r$ is the residue modulo $2$ of the order $o_{p}(-2)$ of $-2$ in $\mathbb{F}_{p}^{\times}$.\vspace{-0.2cm}
			
			\vspace{-0.2cm}
			\item \[ p^{k+\frac12}<\sat((2, -1), \F_p^{2k+1})\leq \sqrt{c_p}\left(\frac32+\frac{r}{2\cdot o_{p}(-2)}\right)(p^k-1)\sqrt{p},\]  
			where $r$ is the residue modulo $2$ of the order $o_{p}(-2)$ of $-2$ in $\mathbb{F}_{p}^{\times}$ and $c_p\leq 3$ is  the same constant depending on $p$ as in Theorem \ref{maingroups}.
			
		\end{enumerate}
	\end{theorem}
	
	\begin{remark}
		The same ideas as in Theorem \ref{mainvectorspace} work to prove analogous results in $\mathbb{Z}_m^k$, when $\gcd(m,6)=1$. Then  $o_p(-2)$ should be replaced by the multiplicative order of $-2$ in the ring $\mathbb{Z}_m$. In the proofs Theorems \ref{dirprod} and \ref{dirprod2} should be used intead of Propositions \ref{linesp} and \ref{linesp2}, respectively. 
	\end{remark}
	
	\begin{theorem}\label{proby}
		For abelian groups $G$ of order $n > 5$ odd, it holds that
		$$\sat(3-\mathrm{AP},G) \leq \sat((1/2,1/2),G) \leq \sqrt{(n-1)\ln{(n-1)}}+ \sqrt{(n-1)}+1.$$    
	\end{theorem}

	
	The paper is organized as follows. 
	In Section 2 we present some preliminary results and useful tools. First we prove lower bounds on the size of  saturating and complete $W$-avoiding sets. Then we show the strength and limitation of direct product constructions, which will enable us to prove the upper bound results for vector spaces (Theorems \ref{main3AP} and \ref{mainvectorspace}). To have upper bounds close to our lower bounds, we will rely on  further avoiding and saturating set constructions in finite fields which are small enough.  Finally, we point out the relation of these results to results concerning caps and $3$-AP covering sequences.\\
	In Section 3 we prove the upper bounds of Theorem \ref{main3AP} by analysing point sets of conics in the respected vector spaces.
	We also prove some general constructions of saturating sets in direct product of groups. Then we deduce the upper bounds of Theorem \ref{mainvectorspace} and Theorem \ref{main3AP} (1).
	Section 4 is devoted to the proof of Theorem \ref{maingroups} and \ref{proby}, where we provide constructions based on numeral systems, additive basis, Sidon sets and random constructions, using tools from additive number theory to design theory. 
	
	\section{Preliminary results}
	
	\subsection{Double counting and direct sum constructions}

	We begin this section by demonstrating some trivial lower bounds on the size of $3$-$\mathrm{AP}$ saturating and $W$-saturating sets.
	
	\begin{prop}
		\label{satbound}
		\begin{enumerate}[\rm(1)]
			
			\item Suppose that $H$ is a $3$-$\mathrm{AP}$ saturating set in the abelian group $G$ of odd order. Then
			\[|H|\geq \sqrt{\frac{2}{3}|G|+\frac{1}{36}}+\frac{1}{6}.\]
			Hence, $\sat(3-\mathrm{AP}, \F_q^n)>0.8164\cdot q^{n/2}$.  \vspace{-0.2cm}
			\item Suppose that $H$ is a $w$-saturating set in the group $G$, where 
			$w\in \mathbb{Z}\times \mathbb{Z}$ 
			or $w=(\lambda_1,\lambda_2)\in \F_q^*\times \F_q^*$  if $G=\F_q^n$. Then 
			\[|H|\geq \lceil \sqrt{|G|} \rceil.\]
			Hence, $\sat(w, \F_q^n)\geq \lceil q^{n/2} \rceil$.  \vspace{-0.2cm}
			\item Suppose that $H$ is a $w$-saturating set in the group $G$, where 
			$w=(1,1)$, or $w=(\frac12, \frac12)$ if $G$ has odd order, or $w=(\lambda,\lambda)$ for some $\lambda \in \F_q^*$ if $G=\F_q^n$. Then
			\[|H|\geq \sqrt{2|G|+\frac14}  -\frac12.\]
			So in this case $\sat(w, \F_q^n)>\sqrt2\cdot q^{n/2}-0.5$
		\end{enumerate}
	\end{prop}
	\begin{proof}
		Part $(1)$. 
		By definition, $\forall x\in G\setminus H$, we have $a\neq b\in H$ s.t.  $x=\frac12 a + \frac12 b$, or $x=2a-b$, or $x=-a+2b$. Thus by double counting, \begin{equation*}\label{eq:1}
			|G\setminus H|\leq 3\binom{|H|}{2},
		\end{equation*}
		from which the lower bound follows.
		
		Part $(2)$.
		Let $w=(\lambda_1, \lambda_2)$. 
		By definition, $\forall x\in G\setminus H$, we have $h\neq h'\in H$ s.t.  $x= \lambda_1h+\lambda_2h'$. 
		Then by double counting, \begin{equation*}\label{eq:11}
			|G|- |H|\leq 2\binom{|H|}{2}.
		\end{equation*}
		After rearranging, we get the desired bound.
		
		Part $(3)$. By double counting,
		\[|G|-|H|\leq \binom{|H|}{2},\]
		and the bound follows after rearranging.
	\end{proof}
	
	
	
	

	
	
	\begin{remark}
		We will see later that the lower bound above for $\sat((2,-1),G)$ is sharp in some cyclic groups, cf. Theorems \ref{diffset} and \ref{gyok3}. Subsection \ref{last} provides further instances when Proposition \ref{satbound}  (2) is sharp. 
	\end{remark}
	
	

	\noindent Next we show that the direct sum construction preserves certain properties concerning saturation and $W$-avoidance. Note however that saturation with respect to $3$-$\mathrm{AP}$s is not preserved.

	\begin{prop}[Avoiding and saturation property in direct products]\label{product}
		\ \\ \vspace{-0.5cm}
		\begin{enumerate}[\rm(1)]
			\item Suppose that $H$ and $H'$ are subsets of the abelian groups $G$ and $G'$, respectively.
			\begin{enumerate}
				\item 
				Assume that $H$ and $H'$ are $3$-$\mathrm{AP}$ free in the corresponding groups. Also, if $G$ (or $G'$) has even order, then assume  that the order of $x-y$ is larger than $2$ for any two distinct $x,y\in H$ ($\in H'$). 
				Then $H\times H'$ is $3$-$\mathrm{AP}$ free in $G \times G'$.
				\item If $H$ and $H'$ are $(1,1)$-avoiding in the corresponding groups and $0_G \notin H$, $0_{G'}\notin H'$ then $H\times H'$ is $(1,1)$-avoiding in $H \times H'$.
			\end{enumerate}
			
			
			\item  Suppose that $H$ and $H'$ are $W$-avoiding subsets of the vector spaces $\F_q^m$ and $\F_q^n$, respectively for some $W \subseteq \F_q^*\times \F_q^*$.  Then $H\times H'$ is $W$-avoiding in $\F_q^m \times \F_q^n$, provided that $\lambda_1+\lambda_2=1$ for all $w=(\lambda_1, \lambda_2)\in W$.
			
			\item Suppose that the set $W$  consists of a single vector $w=(\lambda_1, \lambda_2) \in \F_q^*\times \F_q^*$, provided that $\lambda_1+\lambda_2=1$. If $H\subseteq \F_q^m$ and $H'\subseteq \F_q^n$ are $W$-saturating sets of the corresponding vector space, then $H\times H'$ is also a $W$-saturating set in $ \F_q^{m}\times \F_q^n$.
			\item Suppose that $w=(2,-1)$, or $G$ and $G'$ are of odd order and $w=(1/2,1/2)$. If $H\subseteq G$ and $H'\subseteq G'$ are $w$-saturating sets of the corresponding groups, then $H\times H'$ is also a $w$-saturating set in $G \times G'$.
		\end{enumerate}
	\end{prop} 
	
	Note that the condition $\lambda_1+\lambda_2=1$ holds if and only if  $w_1$, $w_2 $ and   $\lambda_1w_1 + \lambda_2w_2 \in \F_q^n$ are collinear in the affine space $\F_q^n$.
	Observe also that by choosing $w=(2, -1)$ in part (3), 
	the direct product will be $3$-$\mathrm{AP}$ saturating as well.
	
	\begin{proof}
		$(1a)$ Assume to the contrary that there is a $3$-$\mathrm{AP}$: 
		\[(h_1,h_1'),(h_2,h_2'),(h_3,h_3') \in  H\times H'\] 
		with difference $(d,d')\in G\times G'$. Since $(d,d')$ is not the neutral element of the group $G\times G'$, w.l.o.g. we may assume that $d$ is not the neutral element of $G$.
		Then $h_1,h_2,h_3$ is a $3$-$\mathrm{AP}$ of $G$, contradicting the assumption on $H$, or the order of $G$ is even and $h_1+h_3=2h_2$ holds because the size of $\{h_1,h_2,h_3\}$ is $2$, i.e. $h_1=h_3$ and hence the order of $h_1-h_2$ is $2$, a contradiction.
		
		$(1b)$ Assume to the contrary 
		$(h_1,h_1')=(h_2,h_2')+(h_3,h_3')$ for some $h_1,h_2,h_3 \in H$ and $h_1',h_2',h_3' \in H'$. W.l.o.g. we may assume $h_2 \neq h_3$. Then $h_1=h_2+h_3$ are $3$ distinct elements of $H$ contradicting the fact that $H$ is $(1,1)$-avoiding (recall $0_G\notin H$). 
		
		Proof of $(2)$. Assume to the contrary the existence of $w=(\lambda_1, \lambda_2 )\in W$ such that  $(h_1,h_1')= \lambda_1(h_2,h_2')+ \lambda_2(h_3,h_3')$ for some elements of $H\times H'$.  Hence
		\begin{equation*}   
			\begin{cases}
				h_1=\lambda_1h_2+\lambda_2h_3,\\
				h_1'=\lambda_1h_2'+\lambda_2h_3'.\\
			\end{cases}
		\end{equation*}
		Since $(h_2,h_2')\neq (h_3,h_3')$, we may assume w.l.o.g. that $h_2 \neq h_3$.
		We have $h_1=\lambda_1 h_2 + \lambda_2 h_3$ and this is a contradiction if $h_1, h_2, h_3$ are three pairwise distinct elements since $H$ is $W$-avoiding. If we had $h_1=h_2$, then $h_1(1-\lambda_1)=h_3(1-\lambda_1)$ and hence also $h_1=h_3$, a contradiction since $h_2 \neq h_3$ (and the same argument shows $h_1\neq h_3$ as well). 

		Proof of $(3)$. Take any $(g,g') \in (\F_q^{m}\times \F_q^{n}) \setminus (H \times H')$. If $g\notin H$ and $g'\notin H'$, then by the assumption, there exist $h_1,h_2 \in H$ and $h_1',h_2' \in H'$ such that
		\begin{equation*}   
			\begin{cases}
				g=\lambda_1h_1+\lambda_2h_2,\\
				g'=\lambda_1h_1'+\lambda_2h_2',\\
			\end{cases}
		\end{equation*}
		$g,h_1,h_2$ and $g',h_1',h_2'$ are sets of pairwise distinct elements.
		This in turn shows that $$(g,g')=\lambda_1(h_1, h_1')+\lambda_1(h_2, h_2'),$$
		where $(g,g'),(h_1,h_1'),(h_2,h_2')$ are pairwise distinct elements. 
		
		We cannot have $g\in H$ and $g'\in H'$ at the same time, hence w.l.o.g. we may assume $g\in H$ and $g'\notin H'$. 
		Then by the assumption, there exist $h_1',h_2' \in H'$ such that
		\[g'=\lambda_1h_1'+\lambda_2h_2',\]
		$g', h_1', h_2'$ are pairwise distinct elements.
		This in turn shows that $$(g,g')=\lambda_1(g, h_1')+\lambda_2(g, h_2'),$$
		where $(g,g'),(g,h_1'),(g,h_2')$ are pairwise distinct elements. 
		
		The proof of $(4)$ is the same as the proof of $(3)$.\end{proof}
	
	A generalisation of some of the results above will be discussed in Subsection \ref{last}.


	
	
	
	\begin{corollary}
		Direct product of complete $(2,-1)$-avoiding sets is a complete $3$-$\mathrm{AP}$ free set. Moreover,
		$  a((2,-1),G)\cdot a((2,-1),H) \ge a(3-AP,G\times H)$. This highlights the importance of finding complete $(2,-1)$-avoiding sets $A$ in $G$ such that  $|A|\leq \sqrt{|G|} $. 
	\end{corollary}
	

	The previous propositions 
	motivate the distinguishment below.
	
	\begin{definition}
		A complete $3$-$\mathrm{AP}$-avoiding or a complete $(\lambda, 1-\lambda)$-avoiding set $H\subseteq G$ is called \textbf{small} if $|H|\leq  \sqrt{|G|} $. 
	\end{definition}

	\begin{proposition} 
		\label{mirejo}
		Let $H$ denote a $W$-avoiding, $W'$-saturating set in $\F_q^m$.
		\begin{enumerate}[\rm(1)]
			\item Then for each $\lambda \in \F_q^*$ it holds that $\lambda H$ is $W$-avoiding and $W'$-saturating.
			\item If for each $(\lambda_1,\lambda_2) \in W$ it holds that $\lambda_1+\lambda_2=1$, then for each $d\in \F_q^m$, $H + d$ is $W$-avoiding. If for each $(\lambda_1,\lambda_2) \in W'$ it holds that $\lambda_1+\lambda_2=1$, then for each $d\in \F_q^m$, $H + d$ is $W'$-saturating.
			\item Assume $W'=\{(1,1)\}$, $0 \in H$ and $\lambda \in \F_{q}^m \setminus \{0,1\}$. Then for each $x\in \F_q^m \setminus \{0\}$ there exist $a,b \in \lambda H$ such that $x=(1/\lambda) a + (1/\lambda) b$. In particular, $\lambda H$ is $(1/\lambda,1/\lambda)$-saturating.
		\end{enumerate}
	\end{proposition}
	\begin{proof}
		Proof of (1). For some $(\lambda_1,\lambda_2)\in W$ and $a,b,c \in H$, $\lambda_1 \lambda a + \lambda_2 \lambda b = \lambda c $ would imply $\lambda_1 a + \lambda_2 b = c$, a contradiction, which proves that $\lambda H$ is $W$-avoiding. Also, if $x\in \F_q^m \setminus \lambda H$, then $x=\lambda c$ for some $c \notin H$ and hence $c=\lambda_1 a + \lambda_2 b$ for some $(\lambda_1,\lambda_2) \in W'$. It follows that $\lambda H$ is $W'$-saturating.
		
		Proof of (2). If we had $\lambda_1(a+d)+\lambda_2(b+d)=c+d$ for some $a,b,c \in H$ and $(\lambda_1,\lambda_2)\in W$, then also $\lambda_! a+\lambda_2 b=c$, a contradiction. If $x \notin H+d$ then $x=c+d$ for some $c\notin H$ and hence there exists $(\lambda_1,\lambda_2)\in W'$ such that $\lambda_1 a + \lambda_2 b = c$ proving that $H+d$ is $W'$-saturating.
		
		Proof of (3). Take some $x\neq 0$.
		If $x \notin H$, then $x=a+b$ for some $a,b \in H$ and hence $x=(1/\lambda) (\lambda a)+(1/\lambda) (\lambda b)$. If $x \in H$, then $x=(1/\lambda) (\lambda 0)+(1/\lambda) (\lambda x)$. 
	\end{proof}

	
	\begin{proposition}
		Let $q$ be odd. If $H \subseteq \F_q^m$ and $H' \subseteq \F_q^n$ are $(1,1)$-saturating such that the corresponding zero vectors are contained 
		in $H$ and in $H'$, resp.,  then $\frac12(2H \times 2H')$ is $(1,1)$-saturating in $\F_q^m \times \F_q^n$. 
	\end{proposition}
	\begin{proof}
		By (3) of Proposition \ref{mirejo} it follows that $2H$ is $(1/2,1/2)$-saturating in $\F_q^m$ and the same holds in $\F_q^n$ for $2H'$. Then by (3) of Proposition \ref{product} it follows that $(2H)\times (2H')$ is $(1/2,1/2)$-saturating in $\F_q^m \times \F_q^n$. Since this subset contains the zero vector of $\F_q^m \times \F_q^n$, the statement follows again by (3) of Proposition \ref{mirejo}. 
	\end{proof}
	
	\begin{proposition}
		If $H \subseteq \F_q^m$ and $H' \subseteq \F_q^n$ are $(1,-1)$-saturating such that the corresponding zero vectors are contained in $H$ and in $H'$, then $H \times H'$ is $(1,-1)$-saturating in $\F_q^m \times \F_q^n$.  
	\end{proposition}
	\begin{proof}
		Take any $(g,g') \in (\F_q^{m}\times \F_q^{n}) \setminus (H \times H')$. If $g\notin H$ and $g'\notin H'$, then by the assumption, there exist $h_1,h_2 \in H$ and $h_1',h_2' \in H'$ such that
		\begin{equation*}   
			\begin{cases}
				g=h_1-h_2,\\
				g'=h_1'-h_2',\\
			\end{cases}
		\end{equation*}
		$g,h_1,h_2$ and $g',h_1',h_2'$ are sets of pairwise distinct elements.
		This in turn shows that $$(g,g')=(h_1, h_1')-(h_2, h_2'),$$
		where $(g,g'),(h_1,h_1'),(h_2,h_2')$ are pairwise distinct elements. 
		
		We cannot have $g\in H$ and $g'\in H'$ at the same time, hence w.l.o.g. we may assume $g\in H$ and $g'\notin H'$. 
		Then by the assumption, there exist $h_1',h_2' \in H'$ such that
		\[g'=h_1'-h_2',\]
		$g', h_1', h_2'$ are pairwise distinct elements.
		This in turn shows that $$(g,g')=(g, h_1')-(0, h_2'),$$
		where $(g,g'),(g,h_1'),(0,h_2')$ are pairwise distinct elements. 
	\end{proof}

	
	\subsection{Relation to caps}
	
	Let $q$ denote  {\em any} (even or odd) prime power. We describe the relation of the results above  to results concerning caps.
	
	\begin{definition}
		A cap of $\AG(n,q)$ is a point set meeting each line of $\AG(n,q)$ in at most two points. 
		
		\smallskip
		
		A cap is called complete if it cannot be extended to a larger cap.
		
		\smallskip
		
		A saturating set $S$ of $\AG(n,q)$ is a point set with the property that for each $P \in \AG(n,q) ~\setminus~ S$ there exist two distinct points $Q,R \in S$, such that $P$ is incident with the line joining $Q$ and $R$.
	\end{definition}
	
	It is clear from the definitions above that a cap is complete if and only if it is also a saturating set.\\
	The lattice of affine subspaces of $\F_q^n$ is isomorphic to the subspace lattice of $\AG(n,q)$. 
	For results on the (maximum) size of complete caps in $\AG(n,q)$, we refer to \cite{EB, Edel, EPPP, Ty} and the references therein. 
	A related problem is the smallest size of complete caps in finite affine and projective spaces. For the size of small complete caps, the theoretical lower bound is essentially sharp for $q$ even \cite{evencap4, evencap3, evencap2, evencap1} and in some cases also for $q$ odd \cite{completecap}. See also \cite{Anbar, Bartoli, DO, Giuli} and the references therein for small complete caps for $q$ odd.

	\begin{prop}
		Put $W=\{(\lambda_1, \lambda_2) \in \F_q^*\times \F_q^* : \lambda_1+\lambda_2=1\}$. Then we obtain the following.
		\begin{enumerate}[\rm(1)]
			\item Caps of $\AG(n,q)$ and $W$-avoiding sets of $\F_q^n$ are equivalent objects. In particular, by Theorem \ref{product} Part $(2)$ the direct sum of caps is a cap.  
			\item Saturating sets of $\AG(n,q)$ and $W$-saturating sets of $\F_q^n$ are equivalent objects.
			\item Complete caps of $\AG(n,q)$ and complete $W$-avoiding sets of $\F_q^n$ are equivalent objects.
		\end{enumerate}
		If $q=3$, then $W=\{(2,2)\}$, if $q=4$, then 
		$W=\{(i,1+i)\}$. Hence, by Part $(3)$ of Theorem \ref{product}, direct sum of saturating sets is a saturating set and direct sum of complete caps is a complete cap for $q\in \{3,4\}$.  \qed
	\end{prop}

	\subsection{Related results on solving linear equations in algebraic structures}
	
	We summarize some results which are connected to the theme of this paper in the sense that the subject is a subset of a set, in which an equation of special form has no non-trivial solutions. Then we also mention some results of saturation type with respect to an equation.
	
	Most probably the leading examples for the first theme are the Sidon sets.
	A set of elements in an abelian group is called a Sidon set if all pairwise sums of its not necessarily distinct elements are distinct. Equivalently, the equation 
	$a+b=c+d$  has only the trivial solution 
	$\{a, b\}= \{c,d\}$ in the set. Observe that Sidon sets are $3$-$\mathrm{AP}$ free.
	Concerning Sidon sets, Erdős and Turán  observed  \cite{ET} (see also Cilleruelo  \cite{Cille}) that the point set of the parabola in $\F_q\times \F_q$ provides a Sidon set. Cilleruelo showed several further abelian groups admitting Sidon sets of size equal roughly to the square root of the order of the group.  Building on his observations, Huang, Tait and Won showed \cite{Tait} that the largest Sidon sets in $\F_3^n$ are of size $3^{n/2}$, provided that $n$ is even.
	Small complete Sidon sets of abelian $2$-groups are investigated in the recent paper of G. Nagy \cite{GNagy} who showed constructions gained from elipses and hyperbolas in the finite affine plane $\F_q\times \F_q$, and in the papers \cite{Pott} and \cite{red}.
	
	
	
	There is a strong connection between  the case of vector space or cyclic group setting 
	and the case of integer setting, when  a non-trivial solution of a particular equation is forbidden within the interval
	$[1,n]\subset \mathbb{Z}$. 
	For Sidon sets, the papers of Ruzsa \cite{Ruzsa2, Ruzsa} discuss the case of small complete structures, while the  work of Kiss, Sándor and Yang \cite{SandorCs} deals with small saturating sets with respect to $3$-$\mathrm{AP}$s. 
	They used the term \textit{$3$-$\mathrm{AP}$ covering} sequence for another related concept. 
	Let $A_0=\{a_1<\ldots<a_t\}$
	be a set of nonnegative integers such that $\{a_1<\ldots<a_t\}$
	does not contain a 3-term arithmetic progression. A sequence $A=\{a_1,a_2,\ldots\}$
	is called the Stanley sequence of order  $3$ generated by $A_0$, where  the elements outside $A_0$
	are defined by a greedy algorithm as follows. For any $l\ge t$, $a_{l+1}$  is the smallest integer $a>a_l$
	such that $\{a_1,\ldots,a_l\}\cup\{a\}$
	does not contain a $3$-term arithmetic progression. Moreover, a sequence $A$ of non-negative integers is called a $3$-$\mathrm{AP}$-covering sequence if there exists an integer $n_0$
	such that, if $n>n_0$, then there exist $a_1, a_2 \in A$ such that $a_1, a_2, n$
	form a $3$-term arithmetic progression.
	Fang \cite{Fang} made the following improvement.

	\begin{theorem}[Fang \cite{Fang}]\label{fange}
		There is a $3$-$\mathrm{AP}$ covering sequence $S$ of integers such that  $$ \frac{ \mid S \cap [1,n]\mid}{\sqrt{n}}\leq \frac{8}{\sqrt{5}} \approx 3.578$$ holds for all $n$. 
	\end{theorem}
	\noindent This constant cannot be improved to $1.77$ \cite{SandorCs}.
	Note that this  result is strongly related to our main problem since it in turn shows that  $\sat((2,-1), \mathbb{Z}_n )\leq (3.578+o(1))\sqrt{n}$.

	\section{\texorpdfstring{Constructions in $\F_q \times \F_q$ and in other direct products}{Construction in Fq x Fq and in other direct products}} 
	
	In this section $q$ always denotes a prime power, and $p$ denotes a prime.

	\begin{constr}\label{parabola}
		Let $\mathcal{P}$ be the point set of the parabola 
		\[\{(x,x^2) : x \in \F_q\}.\]
	\end{constr}
	
	\begin{theorem}\label{main1} Let $q$ be an odd prime power.
		If $-2$ is not a square in $\F_q$ then Construction \ref{parabola} is a complete $3$-$\mathrm{AP}$ free subset of $\F_q\times \F_q$, hence $a(3-\mathrm{AP}, \F_q^2)\leq q.$
	\end{theorem}
	\begin{remark}
		Note that  Construction \ref{parabola} provides an infinite family of small complete $3$-$\mathrm{AP}$ free subsets of $\F_q\times \F_q$.
		As observed by Erdős and Turán, the parabola construction provides also a (dense) Sidon set, see \cite{Eber, ET}.
	\end{remark} 
	
	
	\begin{proof}[Proof of Theorem \ref{main1}]
		For each $(a,b)\in \F_q^2$, $b\neq a^2$, we prove that one of the following systems of equations have a solution $(x,y)\in \F_q\times \F_q$. 
		\begin{equation*}   
			\begin{cases}
				x+y=2a \\
				x^2+y^2=2b\\
			\end{cases}
			\begin{cases}
				2y-x=a \\
				2y^2-x^2=b\\
			\end{cases}
		\end{equation*}
		This implies that no point $(a,b)$ outside $\cP$ can be added to the construction without violating the  $3$-$\mathrm{AP}$ free property.\\
		Solution for the first system exists if and only if $b-a^2$ is a square in $\F_q$. Indeed, in order to have a common solution, we should get a square value for the discriminant $16a^2-8\cdot(4a^2-2b)=16(b-a^2)$ of  $x^2+(x^2-4ax+4a^2)-2b=0$.\\
		Solution for the second system exists if and only if $2(a^2-b)$ is a square in $\F_q$. Indeed, in order to have a common solution, we should get a square value for the discriminant $16a^2-8\cdot(a^2+b)=8(a^2-b)$ of  $2y^2-(4y^2-4ay+a^2)-b=0$.\\
		If $-2$ is not a square, then either the first, or the second discriminant will be a square, providing a solution to one of the systems.
	\end{proof}
	
	Theorem \ref{main1} in turn implies the upper bound of Theorem \ref{main3AP} (1) on complete $3$-AP free sets in vector spaces once one notes that $-2$ is not a square element in $\mathbb{F}_p$ if and only if it is not a square in $\mathbb{F}_p^{2k+1}$.
	
	Now we show some saturating set constructions.
	
	\begin{constr}
		\label{lines} Let $\langle -2 \rangle$ denote the multiplicative (cyclic) subgroup of $\F_q^\times$, generated by $-2$, where char$(q)\neq 2,3$. Take a set of maximum size in each coset of $\langle -2 \rangle$ for which the equations $-2g=g'$, $4g=g'$ have no solutions within the set. Let $R$ denote the union of these sets.\\
		Let $\cL$ denote the set of point $\cL=\{(0,r) :  r\in \F_q^* \setminus R  \}\cup\{(r,0): r\in \F_q^* \setminus R  \}\subset \F_q \times \F_q$ 
	\end{constr}
	
	The following result is a reformulation of the upper bound of Theorem \ref{mainvectorspace} (1).
	
	\begin{prop}
		\label{linesp}
		Construction \ref{lines} contains $$2(q-1)\frac{o_q(-2)-\lfloor o_q(-2)/3\rfloor}{o_q(-2)}$$ elements and it saturates the $3$-$\mathrm{AP}$s of $\F_q\times \F_q$.  
	\end{prop}
	\begin{cor}
		If $3 \mid o_q(-2)$ then $|\cL|=\frac{4}{3}(q-1)$.
	\end{cor}
	
	\begin{proof}[Proof (of Proposition \ref{linesp})]
		First,  observe that the choice of $R$ ensures that for all $g\in \F_q\setminus \{0\}$, at least two of $g, -2g, 4g$ are admissible coordinates in $\cL$. This implies that at most $$(q-1)\frac{\lfloor o_q(-2)/3\rfloor}{o_q(-2)}$$ elements are contained in $R$. On the other hand, choosing  each element of form $(-2)^{3t-1} $ such that $ 0<3t\leq o_q(-2)$ in $\langle -2 \rangle$ and applying similar rule in each coset yield equality in the bound above. Then the cardinality of points in $\cL$ follows.\\
		Then take any point $(a,b)\in \F_q\times \F_q$, where $a\ne 0\ne b$ and suppose that the addition of $(a,b)$ to the construction does not create a $3$-$\mathrm{AP}$.
		Now take $$(0,2b), (a,b), (2a,0);$$
		$$(-a, 0), (0, b/2), (a,b);$$ and 
		$$(0,-b), (a/2, 0), (a, b)$$ which form three disjoint $3$-$\mathrm{AP}$s consisting of two points of of the axes and $(a,b)$. Here we use the fact  that $o_q(-2)>2$.
		Since at most one element of $\{a/2, -a, 2a\}$ and of $\{b/2, -b, 2b\}$ is contained in $R$, $\cL$ will contain at least $4$ of the points listed above thus together with $(a,b)$, a $3$-$\mathrm{AP}$ would be formed, a contradiction.\\
		Finally, suppose that $a=0$ or $b=0$. Then the addition of $(a,b)$ would again provide at least one $3$-$\mathrm{AP}$, since $(a,b)$ would induce $\frac{q-1}{2}$ pairs $P, P'$ on the axis incident to $(a,b)$ for which $(a,b)$ is the midpoint of $P$ and $P'$, but the number of points on the axis in $\cL$ is larger than $q/2$ thus by the pigeon-hole principle, there would be a pair $P, P'\in \cL$ for which $(a,b)$ is a midpoint, hence the addition of $(a,b)$ is not allowed.
	\end{proof}
	
	\begin{remark}
		If $q$ is a power of the prime $p$ then the multiplicative order of $-2$ in $\F_q^{\times}$ is the same as the multiplicative order of $-2$ in $\F_p^{\times}$.
	\end{remark}
	
	Proposition \ref{linesp} implies directly the upper bound of Theorem \ref{mainvectorspace}  (1) in view of the previous remark if we apply  $q=p^{k}$. 
	
	To get the upper bound when the dimension of the vector space is odd (Theorem \ref{mainvectorspace} (4)), we modify the construction in a way that it $(2, -1)$-saturates the whole space. It enables us to apply the direct sum construction once we have a suitable general upper bound on $\sat((2, -1), \F_p)$.

	\begin{constr}
		\label{lines(2,-1)} Let $\langle -2 \rangle$ denote the multiplicative (cyclic) subgroup of $\F_q^\times$, generated by $-2$, where char$(q)\neq 2,3$. Take a set of maximum size in each coset of $\langle -2 \rangle$ for which the equations $-2g=g'$,  have no solutions within the set. Let $R^*$ denote the union of these sets.\\
		Let $\cL^*$ denote the set of points $\cL^*=\{(0,r) :  r\in \F_q^* \setminus R^*  \}\cup\{(r,0): r\in \F_q^*   \}$. 
	\end{constr}

	\begin{prop}
		\label{linesp2}
		Construction \ref{lines(2,-1)} contains $$2(q-1)-\frac{(q-1)\lfloor o_q(-2)/2\rfloor}{o_q(-2)}$$ elements and it is a $(2, -1)$-saturating  set (and hence $3-\mathrm{AP}$ saturating set)  of $\F_q\times \F_q$.  
	\end{prop}
	
	\begin{proof}
		One should observe that each point $(a,b)$, $a\neq 0\neq b$ is $(2,-1)$-saturated by the pairs  $(-a, 0), (0, b/2)$ and 
		$(0,-b), (a/2, 0)$, and at least one of these pairs will be contained in the construction. The cardinality of $\cL^*$ follows similarly to that of $\cL$ in the proof of Proposition \ref{linesp}.
	\end{proof}
	
	The result above in turn implies the upper bound of Theorem \ref{mainvectorspace} (3). Then,  Theorem \ref{mainvectorspace} (4) follows from the direct sum construction, Proposition \ref{product}, applying it to Construction \ref{lines(2,-1)} with $q=p^k$ and the $(2, -1)$-saturating set construction for $\F_p$, given in the next section (see also Theorem \ref{maingroups} (1)). 
	
	
	
	Along the same lines, one can prove the existence of $3$-AP saturating sets in direct products of abelian groups.
	
	Let $A$ and $B$ denote two abelian groups (written additively) of orders $a$ and $b$, respectively, such that $\gcd(ab,6)=1$.
	
	For  any element $g$ which is not the neutral element of the group, put $D_g=\{g,-2g,4g,-8g,\ldots,-g/2\}$. Since $\gcd(6,ab)=1$, the elements $g,-2g,4g,-8g$ are pairwise distinct, so $|D_g| \geq 4$.
	
	For any element $g$ which is not the neutral element of the group, we denote by $R_g$ a subset of $D_g$ of maximum size such that the equations $x=-2y$ and $x=4y$ cannot be solved within $D_g$. 
	Note that \[\frac13 |D_g| \geq |R_g|=\lfloor\frac13|D_g|\rfloor\geq \frac15 |D_g|.\] 
	Using these notation, we have 
	
	\begin{theorem}
		\label{dirprod}
		Let $A$ and $B$ denote two abelian groups (written additively) of orders $a$ and $b$, respectively, such that $\gcd(ab,6)=1$.
		Then
		\[\cL=\{(r,0_B) : r \neq 0_A,\, r\notin R_g \mbox{ for each } g \in A\} \cup \{(0_A,r) : r \neq 0_B,\, r\notin R_g \mbox{ for each } g \in B\}\]
		is $3$-$\mathrm{AP}$ saturating in $A\times B$.
		On the size of $\cL$ we have
		\[\frac43 (\sqrt{|A\times B|}-1)\leq \frac23(a+b-2) \leq |\cL|\leq \frac45(a+b-2) .\]
		\qed
	\end{theorem}
	
	Note that if $|D_g|$ is the same for each $g\in A$ and $g\in B$ where $g$ is different from the neutral element, then the size of $\cL$ can be expressed via $a$, $b$ and $|D_g|$ for a single $g$. 
	
	This leads to the statement of Theorem \ref{mainvectorspace} (2) by choosing $A=\F_{p}^{k+1}$ and $B=\F_p^k$.

	For any element $g$ of a group $A$ which is not the neutral element of the group, we define $D_g$ as before and we denote by $R^*_g$ a subset of $D_g$ of maximum size such that the equations $x=-2y$ cannot be solved within $D_g$. As before, the size of $D_g$ is at least $4$ and hence
	\[\frac12 |D_g| \geq |R^*_g| \geq \lfloor \frac12|D_g| \rfloor \geq \frac25|D_g|.\]
	Using these notation we have
	
	\begin{theorem}
		\label{dirprod2}
		Let $A$ and $B$ denote two abelian groups (written additively) of orders $a$ and $b$, respectively, such that $a$ is odd and $\gcd(b,6)=1$. 
		Then
		\[\cL^*=\{(r,0_B) : r \neq 0_A,\, r\in A\} \cup \{(0_A,r) : r \neq 0_B,\, r\notin R^*_g \mbox{ for each } g \in B\}\]
		is a $(2,-1)$-saturating (and hence $3-\mathrm{AP}$ saturating) set of $A\times B$.
		On the size of $\cL^*$ we have
		\[a+\frac 12 b- \frac32 \leq |\cL^*|\leq a+\frac35 b - \frac85.\]\end{theorem}
	\qed

	
	

	\section{\texorpdfstring{Complete $3$-$\mathrm{AP}$ free sets and saturation in abelian groups}{Complete 3-AP free sets and saturation in abelian groups}}

	\subsection{Probabilistic upper bound on saturating sets}
	
	
	We start with a general bound using probabilistic arguments and prove Theorem \ref{proby}. While it is off by a logarithmic factor from the lower bound, it is still the best we know in several cases (although not in vector spaces). It also highlights the algebraic nature of constructions meeting or being close to the lower bound.

	\begin{theorem}\label{random_satu} Suppose that the set $H$ saturates the $3$-$\mathrm{AP}$s in  the abelian group $G$ of order $n$, $n>5$ odd, and $H$ is of minimum size. Then we have $$|H|\leq \sqrt{(n-1)\ln{(n-1)}}+ \sqrt{(n-1)}+1.$$    
	\end{theorem}
	
	\begin{proof} The proof follows the probabilistic argument of \cite{Nagy} of the second author, on the size of saturating sets of projective planes.\\ Let $H_0$ be a random subset of $G$ 
		consisting of elements $g\in G$ where each element is chosen independently, uniformly at random with probability $p$. The parameter $p$ will be determined later on.
		Let $H_1$ be the set of elements $g\in G$ which  can be obtained as 
		$2g= h+h'$ or $g= 2h-h'$ for $h,h'\in H_0$.
		Let $X$ denote the random variable which takes the cardinality of $H_0$ and $Y $ denote the random variable which takes the cardinality of $H_1$. 
		
		Then $H_0\cup (G\setminus H_1)$ will provide a set $H$ that saturates the $3$-$\mathrm{AP}$s in $G$. We will determine the value of $p$ which minimise the expected value of $X+n-Y$.
		Clearly, $\E(X)=pn$. \\
		We call a pair $g_1, g_2$\textit{ induced by $g$} if $g_1+g_2=2g$. Hence each element of $G\setminus \{g\}$ is contained in exactly one pair induced by a fixed element $g$. If $g\not \in H_1$ then $H_0$ contains at most one element from each pair induced by $g$.
		Thus $$\pP(g\not \in H_1)<(1-p^2)^{\frac{1}{2}(n-1)}.$$
		By the linearity of expectation, we get 
		$$\E(X+n-Y)<n\left(p+(1-p^2)^{\frac{1}{2}(n-1)}\right).$$
		If $p=\sqrt{\frac{\ln{(n-1)}}{{n-1}}}$, this provides the existence of a set which saturates $3$-$\mathrm{AP}$s and have cardinality at most $$n\left(\sqrt{\frac{\ln{(n-1)}}{{n-1}}}+\left(1-\frac{\ln{(n-1)}}{n-1}\right)^{\frac{n-1}{2}}\right)<\sqrt{(n-1)\ln{(n-1)}}+1+\sqrt{n-1},$$
		taking into account that 
		
		$\sqrt{\frac{\ln{(n-1)}}{{n-1}}}+\frac{1}{\sqrt{n-1}}<1$ and applying the Bernoulli bound $(1-\frac{x}{m})^m<e^{x}$ for $x=-\ln{(n-1)}$ and $m=n-1$.
	\end{proof}
	Actually this argument shows that $\sat((1/2, 1/2), G)\leq  \sqrt{(n-1)\ln{(n-1)}}+ \sqrt{(n-1)}+~1$. The same bound can be easily obtained for $w=(2, -1)$-saturation as well.
	
	\begin{remark}
		Using the Lovász local lemma, one can prove that the probability of  $\pP(g\not \in~ H_1)$ can be  bounded from below by $(1-p^2)^{c(n-1)}$ for some positive constant $c$, which implies that the order of magnitude of a random construction obtained as above will be $\Theta(\sqrt{n \ln {n}})$.
	\end{remark}

	\subsection{\texorpdfstring{Complete $(2,-1)$-avoiding sets of minimum size in cyclic groups via difference sets}{Small complete (2,-1)-avoiding sets in cyclic groups via difference sets}}
	
	
	The Singer difference sets of the cyclic group of order $q^2+q+1$, $q$ a prime power, provide maximal Sidon sets. These constructions inspire the construction below. We use without explicit reference the most well known facts concerning difference sets according to the Handbook of Combinatorial Designs \cite{handbook}.
	
	\begin{theorem}\label{diffset} 
		Put $M=2^{2n}+2^n+1$ and denote by $D'$ a Singer  $(M,2^n+~1,1)$-difference set of the cyclic group $(\mathbb{Z}_M,+)$. 
		Then $D'$ is a complete $3$-$\mathrm{AP}$ free subset of $\mathbb{Z}_M$. Moreover, $D'$ is complete $(2,-1)$-avoiding of size $\lceil \sqrt{M}\rceil$, so its size reaches the lower bound in Proposition \ref{satbound} part $(2)$. 
	\end{theorem}
	
	\begin{remark}
		Note that $M=2^{2n}+2^n+1$ is a prime for $n\in \{1,3,9\}$ and in these cases we obtain  complete $3$-$\mathrm{AP}$ free subsets in the corresponding finite fields of size $2^{2n}+2^n+1$. In general, $n$ needs to be a power of $3$ for this to hold.  Indeed, $M$ can be written as $M= \frac{2^{3n}-1}{2^n-1}$, and if there exists a proper divisor $d\mid 3n$ which is not a divisor of $n$, then $gcd(d,n)<d$. Now, by applying $gcd(2^d-1,2^n-1)=2^{gcd(d,n)}-1$, we get the identity \begin{equation}
			(2^{\gcd(d,n)}-1)\cdot r\cdot \frac{2^{3n}-1}{2^d-1}=(2^n-1)\cdot M,   
		\end{equation} where
		$2^d-1=r\cdot (2^{\gcd(d,n)}-1)$. Hence $r\mid M.$ But on the one hand,\\ $r\leq 2^d-1<2^{2n}<M$, on the other hand, $r>1$ as $\gcd(d,n)<d$.
	\end{remark}
	
	\begin{proof}[Proof of Theorem \ref{diffset}]
		According to the First Multiplier Theorem, for a translate $D$ of $D'$ it holds that $2D=D$. We will show that $D$ is complete $3$-$\mathrm{AP}$ free. Note that this implies that the translates of $D$ are complete $3$-$\mathrm{AP}$ free as well.
		
		
		First we show that $2a=b+c$ cannot hold with $a,b,c\in D$ pairwise distinct elements. Indeed, it would imply $a-b=c-a$, contradicting the fact that $D$ is a difference set. It follows that $D$ is $3$-$\mathrm{AP}$ free.
		
		Since for each $a\in D$ we have also $2a \in D$, in $D \cup \{0\}$ we have the $3$-$\mathrm{AP}$: $\{0,a,2a\}$. This shows $0\notin D$ and that $D$ saturates $\{0\}$.\\    
		Now take any $g\in \mathbb{Z}_M \setminus D$, $g\neq 0$. Then there exist $a,b\in D$  such that $a-b=g$ and since $2D=D$, we have also $a=2c$ for some $c\in D$, that is, $2c=b+g$. We cannot have $c=b$ since in that case $c=g \in D$, a contradiction.
		
		The size of $D$ reaches the lower bound in Proposition \ref{satbound} (2) since $2^n<\sqrt{M}<2^n+1=|D|$.
	\end{proof}
	
	\begin{corollary}
		For $p\in \{7,73,262657\}$, the minimum size of a complete $(2,-1)$-avoiding subset of $\F_p$ is $\lceil \sqrt{p}\rceil$.
	\end{corollary}
	
	
	
	In general one can prove the following, along the same lines.
	
	\begin{prop}
		If $D$ is a $(v,k,\lambda)$-difference set in the group $G$ with numerical multiplier $2$ then $D$ saturates the $3$-$\mathrm{AP}$s.  Moreover, if $\lambda=1$ also holds, then $D$ is a complete $3$-$\mathrm{AP}$ free set.
	\end{prop}
	
	\begin{prop}
		If $D$ is a $(k^2+k+1,k+1,1)$-difference set in $G$, $0 \in D$, then $D$ is complete $(1,-1)$-avoiding of size $k+1$ and hence its size reaches the lower bound in Proposition \ref{satbound}. It follows that $a((1,-1),\mathbb{Z}_{k^2+k+1})=(k+1)$ if $k$ is a prime power.
	\end{prop}
	\begin{proof}
		By definition if $x\in G \setminus \{0\}$ then there exist $y,z \in D$ such that $x=y-z$. 
		
		Assume to the contrary $x=y-z$ for some pairwise distinct $x,y,z \in D$. Then $x-0=y-z$, contradicting the fact that $D$ is a $(k^2+k+1,k+1,1)$-difference set.
		
		The last part follows from the existence of Singer-difference sets. 
	\end{proof}
	
	Note that $0\in D$ can always be obtained since translates of $D$ are difference sets as well.

	\subsection{\texorpdfstring{$(1/2,1/2)$-saturating sets in cyclic groups via additive bases}{(1/2,1/2)-saturating sets in cyclic groups via additive bases}}

	We continue with upper bounds on $\sat(W,\F_p)$ for $W=\{(1/2, 1/2)\}$.
	
	Here we refer to a construction which provides a good upper bound for the solution of the postage stamp problem which is very closely related to finite additive basis, see \cite{Nath, Habsi, Hof}. Recall that the problem was described in the Introduction.

	
	Let $[a,(t), b]$ denote $\{a+t\cdot h: h\in \mathbb{Z}\}\cap [a,b]$.
	\begin{constr}[Mrose, \cite{Mrose}]
		\label{Mr}
		For an arbitrary positive integer $t$ take a set $S$ of $7t+2$ elements as $S=\bigcup_{j=1}^5 A^{(j)}$, where\\ 
		$A^{(1)}:=[0, (1), t],$\\ $ A^{(2)}:= [2t, (t), 3t^2+t],$\\ $ A^{(3)}:=[3t^2+2t, (t+1), 4t^2+2t-1],$\\ $ A^{(4)}:=[6t^2+4t, (1), 6t^2+5t],$\\ $A^{(5)}:=[10t^2+7t, (1), 10t^2+8t]$. 
	\end{constr}
	
	\begin{prop}[Mrose, \cite{Mrose}] \label{Mrose1}
		$(S+S)\supset [0, 14t^2+10t-1]$ holds for the Mrose construction $S$ with parameter $t$. 
	\end{prop}
	
	We apply this classical construction to prove the following upper bound for $\sat(3-\mathrm{AP}, \F_p)$ via $W$-saturation for $W=\{(1/2, 1/2)\}$.
	
	\begin{prop} Suppose that $m$ is odd. Then
		\[\sat(3-\mathrm{AP},\mathbb{Z}_m) \leq \sat((1/2, 1/2), \mathbb{Z}_m)\leq (\sqrt{3.5}+o(1))\sqrt{m}\approx 1.87\sqrt{m}.\]
	\end{prop}

	\begin{proof} Choose the least integer $t$ such that $14t^2+10t-1\geq m$ holds, i.e., $$14t^2+10t-1\geq m\ge 14(t-1)^2+10(t-1)-1.$$ Consider the set $S$  $\pmod m$ obtained in Construction \ref{Mr}. Since $S+S= \mathbb{Z}_m$, we also have $$\left\{\frac{s}{2}+\frac{s'}{2} \  : \  s, s'\in S\subset \mathbb{Z}_m\right\}=\mathbb{Z}_m,$$ since $\gcd(2,m)=1$. 
		Note that for $x\notin S$ and $x=s/2+s'/2$, $s,s' \in S$, we cannot have $s=s'$ and hence $x$ is saturated by two distinct elements of $S$. 
		Hence $S$ is a $(1/2, 1/2)$-saturating set in $\mathbb{Z}_m$ of size $|S|=7t+2$ while $m\ge 14t^2-18t+3> \frac{2}{7}|S|^2-4|S|.$ From this, we get that $\sat((1/2, 1/2), \mathbb{Z}_m)< 7+\sqrt{49+3.5m}.$ 
	\end{proof}

	\subsection{\texorpdfstring{Complete $(2,-1)$-avoiding and $(2,-1)$-saturating sets in cyclic groups}{Complete (2,-1)-avoiding and (2,-1)-saturating sets in cyclic groups}}
	
	We will say that $S$ $(2,-1)$-saturates $[x,y]$ if for each $z\in [x,y]$ there exist $a,b\in S$ such that $z=2a-b$. 
	
	\begin{remark}
		Every integer $1\leq k  \leq \frac43 (4^{n}-1)$ can be written in a unique way as 
		\[k=k_{l}4^{l}+\cdots +k_0 4^0,\]
		where $k_i\in \{1,2,3,4\}$ and $0 \leq l\leq n-1$. 
		This representation of positive integers is known as the bijective base-4 numeral system, see e.g, \cite{Smullyan}.
	\end{remark}

	\begin{constr}\label{z_m}
		
		Let
		\[H_l=\{v_{l-1}4^{l-1}+\cdots +v_04^0 : 
		v_i\in \{2,3\}\text{ for }i=0,1,\ldots ,l-1\},\]
		and
		\[K_l=\{v_{l-1}4^{l-1}+\cdots +v_04^0 : 
		v_i\in \{1,2,3,4\}\text{ for }i=0,1,\ldots ,l-1\},\]
		so $K_l$ is the set of integers with exactly $l$ digits in the bijective base-$4$ numeral system. 
	\end{constr}
	
	Note that $K_l=[\frac13(4^l-1),\frac43(4^l-1)]$, so $|K_l|=4^l$.
	
	The smallest integer of $H_l$ is $\frac23(4^l-1)$, the largest one is $4^l-1$, and $|H_l|=2^l$.
	
	\begin{theorem}
		\label{gyok3} \ \\ \vspace{-0.8cm}
		\begin{enumerate}[\rm(1)]
			\item The set $H_i \cup H_{i+1} \cup \ldots \cup H_{j}$ $(2,-1)$-saturates any subset of $K_i \cup K_{i+1} \cup \ldots \cup K_{j}$ for every pair of positive integers $i\leq j$.
		\end{enumerate}

		
		Given a positive integer $n$, let $m$ denote an integer such that $4^{n-1} < m \leq 4^n$.
		
		\begin{enumerate}[\rm(2)]
			\item If $m=4^n$, then consider the elements of $H_n$ and $K_n$ as representatives for the elements of $\mathbb{Z}_{4^n}$. The $2^n$ elements corresponding to $H_n$ form a complete $(2,-1)$-avoiding set in $\mathbb{Z}_m$.

			\item[\rm(3)] If $\frac13(4^n-1)+1\leq m<4^n$ then consider any interval $[x,y]$,  $H_n \subseteq [x,y] \subseteq K_n$, of size $m$ as a representative for $\mathbb{Z}_m$. Then the elements corresponding to $H_n$ form a $(2,-1)$-saturating set of size less than $\sqrt{3 m}$ in $\mathbb{Z}_m$.
			
			If $\frac23(4^n-1) < m$ then $H_n$ corresponds to a complete $(2,-1)$-avoiding set.
			
			\item[\rm(4)] If $4^{n-1}<m\leq \frac13(4^n-1)$, then let $1\leq k\leq n-1$ be maximal such that 
			\[4^{n-1}< m \leq (4^n-4^{k-1})/3.\]
			Then $S:=H_{k-1} \cup H_{k}\cup \ldots \cup H_{n-1}$ $(2,-1)$-saturates $I:=[\frac13(4^{k-1}-1),\frac13(4^{k-1}-1)+m-1]$ and has size less than $\sqrt{3m}$. Considering $I$ as representatives for $\mathbb{Z}_m$ the same holds for the elements corresponding to $S$. 
		\end{enumerate}
		
	\end{theorem}
	\begin{proof}
		We start with proving $(1)$. It is enough to prove that $H_t$ saturates $K_t$. 
		If $k$ has $t$ digits, then consider the $t$-digit numbers $a$ and $b$ according to Table \ref{tab:my_label1}.
		
		\begin{table}[h!!]
			\centering
			\begin{tabular}{c||c|c|c|c|}
				$k_i$   & 1 & 2 & 3 & 4\\ \hline
				$a_i$  & 2 & 2 & 3 & 3 \\ \hline
				$b_i$  & 3 & 2 & 3 & 2
			\end{tabular}
			\caption{The value of $a_i$ and $b_i, i\leq t-1$, determined by the value of $k_i$.}
			\label{tab:my_label1}
		\end{table}
		
		Note that $k_i=2a_i-b_i$ and hence 
		\[k=(2a_{t-1}-b_{t-1})4^{t-1}+\cdots +(2a_0-b_0) 4^0=\]
		\[2 \sum_{i=0}^{t-1} a_{i}4^i - \sum_{i=0}^{t-1} b_i 4^i=2a-b,\]
		with $a,b \in H_t$. It follows that $b, a,$ and $ k=2a-b$ form a $3$-$\mathrm{AP}$ with difference $a-b$. 
		
		To prove $(2)$ and $(3)$ first we show that $S:=\cup_{i=1}^{\infty} H_i$ is $3$-$\mathrm{AP}$ free in $\mathbb{Z}$.
		Suppose to the contrary that $a<b<c$ are three elements of $S$ forming an arithmetic progression. Assume that $r$ is maximal such that the coefficients $a_r,b_r,c_r$ of $4^r$ are not all equal in the expressions of $a$, $b$, $c$ as above. Then clearly $a_r\leq b_r \leq c_r$. 
		If $a_r=b_r=2$, then $c_r=3$ and $b-a$ is at most $4^{r-1}+\ldots+1$, while $c-b$ is at least $4^r-4^{r-1}-\ldots-1$. It follows that $b-a \neq c-b$. Similar arguments work when $a_r=2$, $b_r=c_r=3$ and when $a_r=0$.
		
		If we consider the set $K_n$ with modulo $4^n$ addition, then  the $(2,-1)$-saturation property clearly holds. We show that the  $(2,-1)$-avoiding property  holds as well. Suppose to the contrary that $a<b<c$ are three elements of $H_n$ forming an arithmetic progression when considered modulo $4^n$.
		Then for some difference $0<d<4^n$ we have $a\equiv c+d \pmod {4^n}$ and 
		either $c-b=d$, or $b-a=d$. 
		From $a\equiv c+d \pmod {4^n}$ it follows that 
		$d \geq 4^n+\frac23(4^n-1)-(4^n-1)=\frac23(4^n-1)+1$ and hence $d=c-b$ and $d=b-a$ are impossible because of $(4^n-1)-\frac23(4^n-1)=\frac13(4^n-1) \geq \max\{c-b,b-a\}$. This proves $(2)$.
		
		The first part of $(3)$ follows from the fact that 
		\[\sqrt{3m}\geq \sqrt{4^n+2}>2^n=|H_n|.\]
		
		To prove the second part suppose to the contrary that $a<b<c$ are three elements of $H_n$ forming an arithmetic progression when considered modulo $4^n$.
		Then for some difference $0<d<m$ we have $a\equiv c+d \pmod m$ and 
		either $c-b=d$, or $b-a=d$. 
		From $a\equiv c+d \pmod m$ it follows that 
		$d \geq m+\frac23(4^n-1)-(4^n-1)=m-\frac13(4^n-1)> \frac13(4^n-1) $ and hence $d=c-b$ and $d=b-a$ are impossible because of $(4^n-1)-\frac23(4^n-1)=\frac13(4^n-1) \geq \max\{c-b,b-a\}$.
		
		To prove $(4)$ assume that $1 \leq k \leq n-1$ is maximal such that 
		$3m \leq 4^n-4^{k-1}$. It follows that 
		\[3m > 4^n-4^k.\]
		First note that $S:=H_{k-1}\cup\ldots \cup H_{n-1}$ has size $2^{k-1}+\ldots+2^{n-1}=2^{k-1}(1+\ldots+2^{n-k})=2^{k-1}(2^{n-k+1}-1)=2^n-2^{k-1}$. Since $S$ $(2,-1)$-saturates $Z:=K_{k-1}\cup \ldots \cup K_{n-1}$ and 
		$I:=[\frac13(4^{k-1}-1),\frac13(4^{k-1}-1)+m-1]\subseteq Z$, it follows that $S$ $(2,-1)$-saturates $I$. We want to show $|S|< \sqrt{3m}$, that is,
		\[2^n-2^{k-1}<\sqrt{3m}.\]
		Clearly, it is enough to prove $4^n+4^{k-1}-2^{n+k}< 3m$, which follows from $4^n+4^{k-1}-2^{n+k}< 4^n-4^k < 3m$. 
	\end{proof}
	
	\subsection{Constructions in abelian groups of composite order}
	\label{last}
	

	\begin{theorem}
		Let $G$ denote a commutative group, $H$ a subgroup of $G$. Put $S=\{a_1,a_2,\ldots,a_s\} \subseteq H$ of size $s$ and $T=\{b_1+H,b_2+H,\ldots,b_t+H\} \subseteq G/H$ of size $t$. Let $w_1, w_2 \in \mathbb{Z}$ such that
		$w_1+w_2=1$ holds. 
		Define
		\[X=\{a_i + b_j : i\in \{1,\ldots,s\},\, j \in\{1,\ldots,t\}\} \subseteq G.\]
		\begin{enumerate}[\rm(1)]
			\item If $S$ and $T$ are $(w_1,w_2)$-saturating in the groups $H$ and $G/H$, respectively, then
			$X$ is $(w_1,w_2)$-saturating in $G$.
			\item Assume that $S$ and $T$ are $(w_1,w_2)$-avoiding in the groups $H$ and $G/H$, respectively, and the order of $x-y$ is not a divisor of $w_1$ in the group $H$ ($G/H$) for each $x,y\in S$ (for each $x,y \in T$). Then $X$ is $(w_1,w_2)$-avoiding in $G$.
			\item Assume that $S$ and $T$ are complete $(w_1,w_2)$-avoiding in the groups $H$ and $G/H$, respectively, and the order of $x-y$ is not a divisor of $w_1$ in the group $H$ ($G/H$) for each $x,y\in S$ (for each $x,y \in T$). Then $X$ is complete $(w_1,w_2)$-avoiding in $G$.
		\end{enumerate}
	\end{theorem}
	
	\begin{proof} First suppose that $S$ and $T$ are $(w_1,w_2)$-saturating sets.
		Take some $c\in G$. Then $c=a+b$ where $a\in H$ and $b+H$ is an element of $G/H$.
		
		By the saturation property, there exist distinct $b_i+H, b_j+H \in T$ such that $w_1( b_i+H) + w_2 (b_j+H) =b+H$. 
		If $w_1b_i + w_2b_j = a'+ b$, for some $a'\in H$, then take some distinct  
		$a_f, a_g \in S$ such that $w_1a_f + w_2a_g=a- a'$.
		Then 
		\[w_1(a_f+b_i)+w_2(a_g+b_j)=c.\]
		If $a_f+b_i=a_g+b_j$, then $a_f-a_g=b_j-b_i \in H$, a contradiction since $i\neq j$. It follows that $X$ saturates $c\in G$. 
		
		Now assume that the conditions of part $(2)$ hold, and 
		\[w_1(a_f+ b_i)+w_2(a_g+b_j)= a_h+b_k\]
		for some elements $a_f+b_i$, $a_g+b_j$, $a_h+b_k$ of $X$. 
		Thus $w_1(b_i+H)+w_2(b_j+H)=b_k+H$ and hence $\{b_i+H, b_j+H, b_k+H\}$ is a set of size at most $2$. 
		If it has size $2$, then we may assume $b_i \neq b_k$. Then the order of $(b_i+H)-(b_k+H)$ is divisible by $w_1$, a contradiction. If $b_i+H=b_j+H=b_k+H$, then $w_1 a_f + w_2 a_g = a_h$. It follows that the size of $\{a_f,a_g,a_h\}$ is at most $2$. If it has size $2$, then we may assume $a_f \neq a_h$. Then the order of $a_f-a_h$ divides $w_1$, a contradiction. Consequently, $a_f=a_g=a_h$ holds and hence
		$a_f+b_i=a_g+b_j=a_h+b_k$. 
		
		The third part is a direct consequence of the first two.
	\end{proof}
	
	In the next result $\mathbb{Z}_r$ is considered as $\{0,1,\ldots,r-1\}$ with operation the usual addition in $\mathbb{Z}$ modulo $r$.
	
	\begin{corollary}
		\label{prodc} 
		Put $S=\{a_1,a_2,\ldots,a_s\} \subseteq \mathbb{Z}_m$ and $T=\{b_1,b_2,\ldots,b_t\} \subseteq \mathbb{Z}_n$. Let $w_1, w_2 \in \mathbb{Z}$ such that
		$w_1+w_2=1$ hold.
		Define
		\[X=\{a_i n + b_j : i\in \{1,\ldots,s\},\, j \in\{1,\ldots,t\}\} \subseteq \mathbb{Z}_{nm}.\]
		\begin{enumerate}[\rm(1)]
			\item If $S$ and $T$ are $(w_1,w_2)$-saturating, then
			$X$ is $(w_1,w_2)$-saturating in $\mathbb{Z}_{nm}$.
			\item Assume that $S$ and $T$ are $(w_1,w_2)$-avoiding, the difference of distinct elements of $S$ is not divisible by $m$, the difference of distinct elements of $T$ is not divisible by $n$. Then 
			$X$ is $(w_1,w_2)$-avoiding in $\mathbb{Z}_{nm}$.
			\item Assume that $S$ and $T$ are complete $(w_1,w_2)$-avoiding, the difference of distinct elements of $S$ is not divisible by $m$, the difference of distinct elements of $T$ is not divisible by $n$. Then
			$X$ is complete $(w_1,w_2)$-avoiding in $\mathbb{Z}_{nm}$.
		\end{enumerate}
	\end{corollary}
	
	\begin{example}
		$\{0,1,2\}$ is complete $(3,-2)$-avoiding in $\mathbb{Z}_9$ and hence  $a((3,-2),\mathbb{Z}_{9^n})=~3^n$.
	\end{example}
	
	\begin{example}
		$\{0,1\}$ is complete $(2,-1)$-avoiding in $\mathbb{Z}_4$ and hence  $a((2,-1),\mathbb{Z}_{4^n})=2^n$, as we already saw in the previous section.
	\end{example}
	
	\section{Concluding remarks and open problems}
	
	In this paper, we proved that in a large family of vector spaces, the minimum size of a complete $3$-$\mathrm{AP}$ free set is equal to a small absolute constant multiple of the  lower bound. However, it remained an open question to decide whether this is true for every vector space $\F_q^n$, $q>2$.
	
	\begin{problem}[Minimum size complete $3$-AP free sets] \ \\ Is it true that $$a(3-AP, \F_q^n)<C\cdot \sqrt{q^n}$$ for an absolute constant $C$, that is, the natural lower bound is tight up to a constant factor?
	\end{problem} 
	Concerning cyclic groups, we pose the following
	\begin{problem} Is it true that
		$a(3-AP,\mathbb{Z}_m)< a((2,-1), \mathbb{Z}_m)<\sqrt{c_m \cdot m}$ holds for $c_m\in [1, 1.5]$, a constant depending only on $m$, for all large enough values of $m$? 
	\end{problem}
	The first inequality follows from the definition (see Remark \ref{rem}), while we proved the second inequality for a dense set of natural numbers $m$ in Theorem \ref{maingroups}. It would be also interesting to see an improvement on the constant $c_m$.
	
	\section*{Acknowledgement}
	The authors would like to thank the referees for their helpful suggestions.\\
	This work was supported by the Italian National Group for Algebraic and Geometric Structures and their Applications (GNSAGA--INdAM). Both authors acknowledge the partial support of the Hungarian Research Grant (NKFI) K 124950. The first author is supported by the J\'anos Bolyai Research Scholarship of the Hungarian Academy of Sciences. The second author is supported by the Hungarian Research Grant (NKFI) No. PD  134953 and by the University Excellence Fund of Eötvös Loránd University.

\end{document}